\documentclass{article}
\usepackage{amsmath,amssymb,tensor,dsfont}
\usepackage{enumitem}
\usepackage{todonotes}

\usepackage[colorlinks=true,
	linkcolor=blue,
	citecolor=green!50!black,  
	urlcolor=teal]{hyperref}

\usepackage[sortcites=true,giveninits=true]{biblatex}
\usepackage{amsthm}
\theoremstyle{plain}

\newtheorem{lemma}{Lemma}
\newtheorem{thm}{Theorem}

\theoremstyle{definition}
\newtheorem{defn}{Definition}
\newtheorem{rmk}{Remark}
\newtheorem{ex}{Example}
\newtheorem{conv}{Convention}

\usepackage{mathrsfs}

\renewcommand{\P}{\ensuremath{\mathsf{P}}}
\newcommand{\R}{\ensuremath{\mathbb{R}}}
\newcommand{\tr}{\ensuremath{\mathrm{tr}}}
\newcommand{\grad}{\ensuremath{\mathrm{grad}}}
\renewcommand{\div}{\ensuremath{\mathrm{div}}}

\usepackage{stmaryrd}

\title{Second-order superintegrable systems and Weylian geometry}
\author{
	\textsc{Andreas Vollmer}\\[.003cm]
	\textit{\scriptsize Universität Hamburg, Fachbereich Mathematik}\\[-.13cm]
	\textit{\scriptsize Bundesstr. 55, 20146 Hamburg, Germany}\\[-.071cm]
	\texttt{\small andreas.vollmer@uni-hamburg.de} }
\date{\today}

\begin{document}
	
	\maketitle
	
\begin{abstract}
	\noindent Abundant second-order maximally conformally superintegrable Hamiltonian systems are re-examined, revealing their underlying natural Weyl structure and offering a clearer geometric context for the study of Stäckel transformations (also known as coupling constant metamorphosis). This also allows us to naturally extend the concept of conformal superintegrability from the realm of conformal geometries to that of Weyl structures.
	It enables us to interpret superintegrable systems  of the above type as semi-Weyl structures, a concept related to statistical manifolds and affine hypersurface theory.\bigskip
	
	\noindent\textsc{MSC2020:} 
	53C18; 
	37J35, 
	70G45. 
	
	\noindent\textsc{Keywords:}
	Weylian manifold, second-order superintegrable system. 
\end{abstract}

\section{Introduction}\label{sec:introduction}

Let $M$ be a simply connected, orientable smooth manifold of dimension $n$. A Riemannian metric~$g$ on $M$ is called a \emph{metric}. Its associated \emph{conformal metric} on $M$ is
\[
c = \left\{ \Omega^2 g~|~\Omega\in\mathcal C^\infty(M), \Omega\ne0 \right\}\,.
\]
We then call $(M,c)$ a \emph{conformal manifold}. We say that it is \emph{flat}, if there exists $h\in c$ such that $h$ is flat, in which case all elements of $c$ are \emph{conformally flat}.
The cotangent space $T^*M$ of $M$ naturally carries a symplectic structure $\omega=-d\theta$ thanks to the tautological $1$-form $\theta$.
It allows us to define the natural vector field $X_F\in\mathfrak X(T^*M)$ associated to a function $F:T^*M\to\R$ via
\[
\omega(X_F,-) = dF\,.
\]
Furthermore, this allows one to define a natural Poisson structure on $T^*M$, with the Poisson bracket $\{-,-\}:\mathcal C^\infty(T^*M)\times \mathcal C^\infty(T^*M)\to \mathcal C^\infty(T^*M)$ defined by
\[
\{F_1,F_2\}=\omega(X_{F_1},X_{F_2})\,.
\]
In canonical Darboux coordinates $(x,p)$, it takes the form
\[
\{F_1,F_2\}=\sum_{i=1}^n \left( \frac{\partial F_1}{\partial x^i}\frac{\partial F_2}{\partial p_i}-\frac{\partial F_1}{\partial p_i}\frac{\partial F_2}{\partial x^i}\right)\,.
\]
Now let $V\in\mathcal C^\infty(M)$. We call the naturally defined function $H:T^*M\to\R$,
\[
H(x,p)=g^{-1}_x(p,p)+V(x)
\]
a \emph{Hamiltonian} on $M$. For the purposes of the present paper we also introduce multi-Hamiltonians: given a vector space $\mathcal V$ (over $\R$) in the space of functions on $M$, we say that
\[
\mathcal H=\{ g^{-1}_x(p,p)+V(x)\in\mathcal C^\infty(T^*M)~|~ V\in\mathcal V \}
\]
is the naturally associated \emph{multi-Hamiltonian}.
We say that a multi-Hamiltonian is \emph{non-degenerate} if $\dim(\mathcal V)=n+2$, cf.\ \cite{KMK2007,KKM-1,KKM-2,KKM-3} for instance.

We call two multi-Hamiltonians $\mathcal H$ and $\mathcal H'$ \emph{conformally related}, if there is a function $\Omega\in\mathcal C^\infty(M)$ such that
\[
	\mathcal H'=\{ \Omega^{-2}H~|~H\in\mathcal H \}\,.
\]
As one quickly verifies, this poses a strong restriction on the function $\Omega$. Indeed, an arbitrary choice of $\Omega$ clearly does not preserve integrals of the motion. This motivates the introduction of so-called \emph{conformal integrals}, i.e.\ functions $F:T^*M\to\mathbb R$ that are only required to remain constant along Hamiltonian trajectories with $H=0$ (details will be given later).

\subsection{Second-order superintegrable systems}

Non-degenerate multi-Hamiltonians that are \emph{superintegrable}, i.e.\ that admit more than $n$ independent functions $T^*M\to\mathbb R$ that Poisson commute with the (full) multi-Hamiltonian, have received considerable attention in the literature, e.g.\ \cite{KKM2018,Kress&Schoebel,Kress07,Capel_phdthesis,KSV2023,PKM2013}.
Here, we will limit ourselves to \emph{maximally} superintegrable systems of \emph{second order}. By this we mean (multi-)Hamiltonians that commute (with respect to the Poisson bracket) with $2n-2$ additional integrals of the motion. Let $H\in\mathcal H$. For second-order maximal superintegrability, we require functions $F^{(\alpha)}:T^*M\to\mathbb R$, $1\leq\alpha\leq 2n-2$, with $\{H,F^{(\alpha)}\}=0$ such that
\[
	F^{(\alpha)}(x,p)=\sum_{i,j=1}^n K^{(\alpha)\,ij}(x)p_ip_j+W^{(\alpha)}(x),
\]
for functions $K^{ij}$ and $W$ on $M$, and such that the tuple $(H,F^{(1)},\dots,F^{(2n-2)})$ is functionally independent. We mention the well-known fact that here only the functions $W^{(\alpha)}$ depend on the choice $H\in\mathcal H$. Indeed, the coefficients $K^{(\alpha)\,ij}$ are linked to Killing tensor fields, which allows one to formulate the conditions for second-order (maximal) superintegrability geometrically and independent of the choice of the $F^{(\alpha)}$. This will be discussed in more detail in Section~\ref{sec:abundant}, along with proper definitions of the structures we consider in this paper. In the present section, our focus is to give a brief review of some results in the theory of second-order maximally superintegrable systems.

Historically, the study of superintegrable systems has its roots in Kepler's work on the two-body problem, i.e.\ the motion of a planet around a central star. This system is a $3$-dimensional maximally superintegrable system: the angular momenta and the Laplace-Runge-Lenz vector provide $2n-2=4$ functionally independent integrals of the motion in addition to the energy. According to \cite{Evans1990}, the close relationship of second-order superintegrability and the existence of several systems of separation coordinates (of the Hamilton-Jacobi equation) was first noted by Sommerfeld in \cite{Sommerfeld}. Separation systems indeed played a major role in the understanding of second-order maximal superintegrability.
Using this relationship, a classification of $2$-dimensional systems (i.e.\ second-order maximally superintegrable systems with two degrees of freedom) was achieved for spaces of constant curvature in \cite{Evans1990,KKPM2001}, see also \cite{KKM2018,KKM-1,KKM-2}.
For $3$-dimensional systems, partial classification results exist, cf.\ \cite{KKM-3,KKM-4,CK14,Capel&Kress,Capel_phdthesis} and \cite{N1,N2,N3,N4,N5}, for instance.
Cross-relations have been found to so-called quadratic algebras. \.{I}n\"on\"u-Wigner algebra contractions have been used to construct examples, e.g.\ \cite{KM2014,Capel&Kress&Post}. Non-degenerate second-order superintegrability has also been related to hypergeometric orthogonal polynomials organised in the Askey-Wilson scheme \cite{PKM2013,PKM2011}.

The concept of non-degeneracy has played a fruitful role in many of these developments. Non-degenerate multi-Hamiltonians appear to have been introduced in \cite{KMP2000}. A classification of non-degenerate second-order maximally superintegrable systems on conformally flat space is to date achieved in dimensions~$2$ and~$3$, see \cite{Kress07,KKMW2003,KKPM2001,KKM-2,Kress&Schoebel} and \cite{Capel_phdthesis,Capel&Kress}, respectively.
In the classification of $3$-dimensional systems, an important ingredient is the so-called ``$(5\longrightarrow6)$-theorem'' (Theorem 3 in \cite{KKM-3}), which states that $3$-dimensional non-degenerate second-order maximally superintegrable systems admit an additional linearly independent integral of the motion. 
More generally, one can define the concept of \emph{abundantness} for non-degenerate superintegrable systems, cf.\ \cite{KSV2023}: it formalizes the property of a non-degenerate system (which has $2n-1$ functionally independent integrals of the motion, including the Hamiltonian) to admit $n(n+1)/2$ linearly independent integrals of the motion.
To the best knowledge of the author, all non-degenerate systems known to date are of this particular type.

In recent years, considerable advances have been made in the understanding of the geometric structures underlying maximally superintegrable systems.
For instance, for two-dimensional non-degenerate systems in the Euclidean plane, \cite{Kress&Schoebel} establishes an algebraic geometric classification, associating a unique planar line triple arrangement to every superintegrable system. For three-dimensional systems, \cite{Capel&Kress,Capel_phdthesis} provide an explicit description of the conformal classes of these systems using algebraic varieties formed from representations of the conformal group.
To facilitate an extension of the classification to higher dimensions, a geometric characterisation of non-degenerate systems has recently been put forward, cf.\ \cite{KSV2023,KSV2024,KSV2024_bauhaus}, proposing a tensorial formalism that makes higher dimensions accessible to an efficient investigation. It revealed that the Riemannian metrics underlying non-degenerate superintegrable systems are often Hessian metrics \cite{AV2025,CV2025}. Examples on spaces of constant sectional curvature arise from solutions to the (possibly non-flat) Witten-Dijkgraaf-Verlinde-Verlinde (WDVV) equation \cite{KSV2024,Vollmer2024}.

\subsection{Stäckel transformations, coupling constant metamorphosis and conformal superintegrability}

We consider the classical Maupertuis-Jacobi transformation. Let $H(x,p)=g^{-1}_x(p,p)+V(x)$ be a Hamiltonian with potential $V$, and define a new (free) Hamiltonian by
\[
	\tilde H(x,p):=\frac{H(x,p)}{V(x)}.
\]
Then, by a classical result, the Hamiltonian trajectories of $\tilde H$ are the geodesics of~$H$ after a transformation of the time coordinate to $\tilde t$, with $d\tilde t=V(x(t),p(t))\,dt$, cf.\ \cite{Tsiganov2001}.
A similar, but extended transformation has been introduced for integrable systems that admit separable coordinates by Boyer, Kalnins and Miller in \cite{BKM1986}. 
In our context, we may view it (in a simplified manner\footnote{The transformation was originally formulated in terms of separation coordinates and St\"ackel multipliers. Thus, the original statement considers collections of $n$ mutually commuting integrals of the motion, which are transformed into commuting integrals of the motion.}) as follows: consider a pair of Hamiltonians $H_V(x,p)=g^{-1}_x(p,p)+V(x)$ and $H_U(x,p)=g^{-1}_x(p,p)+U(x)$ with integrals of the motion $F_V(x,p)=A(x,p)+W_V(x)$ respectively $F_U(x,p)=A(x,p)+W(x)_U$ of second order, where $A=A(x,p)$ is homogeneous of quadratic order in the momenta $p$. Note that $A$ is shared by both integrals. Then the new Hamiltonian $\tilde H=U^{-1}H_V$ is called the Stäckel transform of $H_V$ via $U$. It admits the second-order integral $\tilde F=F_V+(1-W_U)\tilde H$.

A related concept is coupling constant metamorphosis, which was introduced by Hietarinta et al.\ in \cite{HGDR1984} for integrable Hamiltonians of the form $H_\mu(x,p)=\sum_{i,j=1}^n g^{ij}(x)p_ip_j+V(x)+\mu W(x)$ with potentials $V$ and $W$ and a coupling parameter $\mu$. If $F_\mu$ is an integral of $H_\mu$, then let
\[
	\tilde H_\nu(x,p)=W(x)^{-1}\left( \sum_{i,j=1}^n g^{ij}(x)p_ip_j+V(x)+\nu \right)
\]
where $\nu$ is the new coupling parameter. It follows that $\tilde F_\nu:=F_{\tilde H_\nu}$ is an integral of $\tilde H_\nu$. Note that here we replace the coupling parameter $\mu$ by $\nu$, i.e.\ we swap, conceptually, the roles of the coupling parameter and the energy.
To avoid misunderstandings, we mention that Stäckel transformations and coupling constant metamorphosis are in general not the same. However, they coincide in the case of second-order maximally superintegrable systems, which we consider here. For more information on the relationship of the two transformations, we refer the interested reader to \cite{Post10}.

Note that Stäckel transformations perform a conformal rescaling of the Hamiltonian, $H\longrightarrow\frac{H}{\Omega^2}$, where $\Omega^2$ however has to be a very specific choice of function. Indeed, rescalings of the above type obviously do not, in general, preserve superintegrability of a given Hamiltonian.
A very natural question thus is whether a broader framework exists in which we may carry out rescalings via a generic function $\Omega$.  The answer is yes: generic conformal rescalings do preserve the so-called second-order (maximal) \emph{conformal} superintegrability, in which integrals of the motion are replaced by conformal integrals, i.e.\ functions that are constant along Hamiltonian trajectories \textsl{with $H=0$}. (Note that $H$ is an integral of the motion and hence $H=0$ is preserved along trajectories.) A precise definition of conformally superintegrable systems is given in Section~\ref{sec:abundant}.
Conformal superintegrability aligns very naturally with the (aforementioned) tensorial framework developed in \cite{KSV2024,KSV2024_bauhaus}. Indeed, the first of these references argues that these systems should be considered as structures on conformal manifolds (called \emph{c-superintegrability} therein).
\medskip

The purpose of the present paper is to clarify the concept of c-super\-integra\-bility by shedding light onto it from the angle of \emph{Weylian geometry}, closing a gap in the literature.
We explain how an abundant superintegrable system gives rise to a Weyl structure.
We furthermore extend the definition of abundant superintegrable systems to Weylian manifolds, including the case of $2$-dimensional systems. (Note that, while the ideas of~\cite{KSV2024} work in general dimension, there is a focus on $n\geq3$  regarding the abundant structural equations and c-superintegrability.)
As we are going to see, cf.\ Definitions~\ref{defn:weylian.abundant.structure} and~\ref{defn:abundant.structure.weylian.mfd} below, subtleties arise if the underlying Weylian geometry is fixed in advance. In contrast, the abundant structure itself defines a preferred Weylian structure on $(M,c)$, which may differ from a given one. We characterise this difference by invariant data. In Section~\ref{sec:structural.n>2}, we review and improve the invariant structural equations for dimensions $n\geq3$, cf.\ \cite{KSV2024}.
In Section~\ref{sec:structural.n>2}, we determine the analogous invariant equations for the structural equations in dimension~$n=2$, based on the structural equations given in~\cite{KSV2024_bauhaus}, which differ significantly from those in higher dimensions. We conclude the paper in Section~\ref{sec:semi-weylian} showing that abundant superintegrable systems can be realised as semi-Weyl structures. These structures have received some attention in the context of statistical manifolds and affine hypersurface geometry, e.g.\ \cite{BN2022,Matsuzoe2001,Matsuzoe2010}, for instance.\medskip

The paper is organised as follows: In the next section, we introduce the main objects of the paper, namely Weylian manifolds and abundant superintegrable systems in Sections~\ref{sec:weylian} and~\ref{sec:abundant}, respectively.
In Section~\ref{sec:main}, the main discussion of the paper will take place. Specifically, in Section~\ref{sec:structural.n>2}, we discuss the case of systems of dimension $n>2$, while Section~\ref{sec:structural.n=2} is devoted to the case of $2$-dimensional systems, for which the underpinning structural equations are considerably different from higher dimensions.
Finally, we use so-called semi-Weyl structures (a generalisation of Weyl structures) to encode abundant superintegrable systems, see Section~\ref{sec:semi-weylian}.

\section{Preliminaries}

In this section, we introduce the main structures used in this paper, namely Weylian manifolds (Section \ref{sec:weylian}) and abundant non-degenerate superintegrable systems (Section \ref{sec:abundant}), which geometrically can be formalized as abundant structures (Section \ref{sec:sbundant.structure}).

\subsection{Weylian manifolds}\label{sec:weylian}

The purpose of the present section is to introduce Weylian manifolds, also known as \emph{Weyl manifolds}, e.g.\ \cite{Folland1970,Norden1976,CP1999,BFM2023}.
Following \cite{MS2020}, we use the attribute Weylian in order to avoid confusion with the concept of Weyl manifolds in Cartan geometry, cf.~\cite{CM2023, CS2009}.
The ensuing twofold definition follows \cite{MS2020}, see also the references therein.

\begin{defn}\label{defn:Weyl.structure}~
	\begin{enumerate}[label=(\roman*)]
		\item A \emph{Weyl structure} is given by a triple $(M,c,D)$ consisting of a differentiable manifold $M$, a conformal class $c$ of Riemannian metrics on $M$, and a torsion-free affine connection $D$ satisfying the condition, for $g\in c$,
		\begin{equation}\label{eq:Weylian.connection}
			D g=-\theta_g\otimes g
		\end{equation}
		for some differential 1-form $\theta_g$ depending on $g$.
		\item A \emph{Weylian structure} $(M,\Phi)$ is a differentiable manifold $M$ together with a \emph{Weylian metric} $\Phi$ on $M$, i.e.\ an equivalence class of pairs $(g,\varphi)$ consisting of a Riemannian metric $g$ and a real-valued differential $1$-form~$\varphi$, identified under
		\[ (g,\varphi)\sim(\tilde g,\tilde \varphi)\quad:\Leftrightarrow\quad \tilde g=\Omega^2g\,,\quad \tilde\varphi=\varphi-d\ln|\Omega|\,, \]
		for some function $\Omega\in\mathcal C^\infty(M)$.
	\end{enumerate}
\end{defn}

\begin{rmk}\label{rmk:Weylian.correspondence}
	The two definitions are equivalent as a Weylian metric $\Phi$ gives rise, for each representative $(g,\varphi)\in\Phi$, to a connection $\nabla^{(g,\varphi)}$ satisfying
	\[ \nabla^{(g,\varphi)}g+2\varphi g = 0\,. \]
	This statement goes back to Weyl \cite{Weyl1918,WeylGes}. We then find ($X,Y,Z\in\mathfrak X(M)$)
	\begin{align*}
		\nabla_Z^{(g,\varphi)}(\Omega^2g)(X,Y)
		&=2\Omega^2\,(d\ln|\Omega|)(Z)\,g(X,Y)-2\varphi(Z)\Omega^2\,g(X,Y)
		\\
		&=-2\,(\varphi(Z)-(d\ln|\Omega|)(Z))\,g(X,Y),
	\end{align*}
	i.e.\ we can set $D=\nabla^{(g,\varphi)}$, yielding a Weyl structure $(M,[g],D)$. However, note that the Weyl structure is not unique: in fact, for any $1$-form $\alpha\in\Omega^1(M)$, $(M,[g],D')$ with $D'_XY:=D_XY+\alpha(X)Y$ is also a Weyl structure. On the level of Weylian structures, this corresponds to a shift of $\varphi$ by $\alpha$, and therefore we may identify Weylian structures and Weyl structures.
	For more details, see \cite{MS2020}.
	In the remainder of the paper, we speak simply of a \emph{Weylian manifold} whenever we do not wish to specify whether we consider a given setup as in Definition~\ref{defn:Weyl.structure} as a Weyl or a Weylian structure.
\end{rmk}

We now introduce some basic properties of Weylian manifolds.

\noindent\textsl{Flatness of a Weylian manifold:} we say that a Weyl structure is \emph{flat}, if the underlying conformal metric $c$ is flat. Equivalently, a Weylian structure is \emph{flat}, if there is $(g,\varphi)\in\Phi$ such that $g$ is flat.\smallskip

\noindent\textsl{Exactness and closedness of a Weylian manifold:}
A Weyl structure is \emph{exact}, if the connection $D$ is the Levi-Civita connection of some metric in $c$, and \emph{closed}, if this property holds in the neighbourhood of each point. Equivalently, these properties hold, if~$\theta_g$ is exact or closed, respectively, for some (and then any) element of $c$, cf.\ \cite{MP2024,BFM2023}. We define exactness and closedness for  Weylian structures accordingly.

\subsection{Abundant superintegrable systems}\label{sec:abundant}

In Section~\ref{sec:introduction}, we have mentioned the role of conformal transformations (presented as Stäckel transformations or coupling constant metamorphosis) in the context of second-order superintegrability. We have also indicated that generic conformal rescalings of the Hamiltonian, while not preserving superintegrability, do preserve conformal superintegrability. We also recall that Weyl manifolds allow for conformal changes of the metric.
With this in mind, it is obvious that the natural setting for the study of the relationship of Weylian geometry and second-order superintegrability is that of second-order conformally superintegrable systems.

\begin{conv}
	Note that in the following, in order to keep the discussion concise, we will usually drop the attributes `maximal', `second-order' and `conformal' when discussing second-order maximally conformally superintegrable systems, as there is no risk of confusion.
\end{conv}

A \emph{superintegrable system}, or, more precisely speaking, a second-order maxi\-mally conformally superintegrable system, is a Riemannian manifold $(M,g)$ with a natural Hamiltonian $H:T^*M\to\R$ admitting $2n-2$ additional functions $F^{(\alpha)}:T^*M\to\R$, $1\leq\alpha\leq 2n-2$, such that (with $F^{(0)}:=H$)
\begin{enumerate}[label=(\roman*)]
	\item for each $1\leq\alpha\leq 2n-2$ there are functions $K^{ij}$ and $W$ on $M$ such that
	\begin{equation}\label{eq:second-order}
		F^{(\alpha)}(x,p)=\sum_{i,j=1}^n K^{ij}(x)p_ip_j+W(x)\,.
	\end{equation}
	\item the functions $F^{(\alpha)}$, $1\leq\alpha\leq 2n-2$, satisfy
	\begin{equation}\label{eq:integral}
		\{F^{(\alpha)},H\}=2\varrho^{(\alpha)}\,H
	\end{equation}
	for suitable $\varrho^{(\alpha)}:T^*M\to\R$, respectively. We require these $\varrho^{(\alpha)}$ to be polynomial in the momenta.
	\item the functions $(F^{(\alpha)})_{0\leq\alpha\leq 2n-2}$ are functionally independent.
\end{enumerate}
The first point in the above list is the \emph{second-order} property, and the number $2n-2$ indicates the maximality of the superintegrable system. Equation~\eqref{eq:integral} implies that the functions $F^{(\alpha)}$, $0\leq\alpha\leq2n-2$, are constant when restricted to Hamiltonian trajectories on the zero locus of $H$, i.e.~the solution curves $\gamma$ of Hamilton's system of equations,
\begin{equation}\label{eq:hamilton's}
	\dot\gamma = X_H\circ\gamma\,,
\end{equation}
that in addition satisfy $H(\gamma(t))=0$.

\begin{rmk}
	Consider \eqref{eq:integral} and recall that we use canonical Darboux coordinates $(x,p)$. As indicated, for each $\alpha$, $\varrho^{(\alpha)}$ is a function on $T^*M$ which is polynomial in the momenta $p$. It is easily checked that the left hand side of \eqref{eq:integral} is a cubic polynomial in momenta (with coefficients that depend on the position $x$). In fact, it has two homogeneous components (in the momenta), namely a cubic and a linear one. The same has to be true for the right hand side. Therefore, it follows that $\varrho$ has to be linear in $p$, and hence there is a $1$-form $\rho\in\Omega^1(M)$ with $\varrho=\rho(p)$.
	A straightforward computation shows that the cubic homogenous (in momenta) part of the polynomial condition \eqref{eq:integral} is equivalent to the condition of a \emph{conformal Killing tensor},
	\[
	\nabla_XK(X,X)=\rho(X)g(X,X)\qquad\qquad\forall X\in\mathfrak X(M),
	\]
	where $K=\sum_{i,j=1}^n(\sum_{a,b=1}^n g_{ia}K^{ab}g_{bj})\,dx^i\odot dx^k$, and where
	$$ \rho=\frac1{n+2}\left( 2\div_g(K)+d\tr_g(K) \right). $$
	The linear homogeneous component, on the other hand, can be rewritten as
	\begin{equation}\label{eq:pre-BD}
	dW = \hat K(dV)+V\,\rho\,,
	\end{equation}
	where $\hat K$ denotes the endomorphism naturally associated to $K$ (via the inverse of the metric $g$). 
\end{rmk}

Consider a conformal integral, i.e.\ a function $F:T^*M\to\mathbb R$ such that $\{H,F\}=\varrho\,H$ as above. Denote the conformal Killing tensor associated to $F$ by $K$. It follows that also $F'=F-\frac1n\,\mathrm{tr}_g(K)\,H$ is a conformal integral,
\[
	\{H,F'\}=\{H,F\}-\frac1n\,\{H,\tr_g(K)\}\,H=\left(\varrho-\frac1n\,(d\tr_g(K))(p)\right)\,H=:\varrho'\,H\,.
\]
As a result, the quadratic terms of $F'$ correspond to a trace-free conformal Killing tensor.
This allows us to agree to the following convention.

\begin{conv}
	Without loss of generality, we assume from now on that $K$ (for each $F^{(\alpha)}$, $1\leq\alpha\leq2n-2$, be trace-free. This is no restriction: modifying the trace of a conformal Killing tensor, we again obtain a conformal Killing tensor.
\end{conv}

\begin{conv}
	We will tacitly use a circumflex diacritic to denote, for a symmetric tensor field $Q\in\Gamma(\mathrm{Sym}^{k+1}(T^*M))$, the (via $g$) naturally associated section of $\mathrm{Sym}^{k}(T^*M)\otimes TM$, whenever the underlying metric $g$ is clear.
\end{conv}

We denote by $\Pi_{(ijk)\circ}$ the projector onto the totally symmetric (in $i,j,k$) and trace-free (w.r.t.\ contraction in any pair of $i,j,k$) part of a tensor field (which may have further indices over which no symmetrisation nor trace-freeness is enforced). The circle $\circ$, in particular, will mean trace-freeness. Consequently, the projector onto the totally tracefree part is denoted by $\Pi_\circ$. By $\Gamma(\mathrm{Sym}^2_\circ(T^*M))$ we denote the space of sections in the symmetric, trace-free tensors of rank two.

We return to \eqref{eq:pre-BD}. Our aim is now to explain why the concept of a second-order maximally superintegrable system is well-defined for a multi-Hamiltonian system $\mathcal H$.
To this end, recall that in the integrals $F$, of the form $F=\sum_{i,j=1}^n K^{ij}p_ip_j+W$, only the scalar part~$W$ depends on the choice $H\in\mathcal H$. Indeed, the condition of a conformal Killing tensor only relies on the metric $g$.
Hence, having fixed $H\in\mathcal H$, we may integrate \eqref{eq:pre-BD} for $W$, provided that the integrability condition for \eqref{eq:pre-BD} holds. It is obtained by applying the differential (exterior derivative) to~\eqref{eq:pre-BD}. We find the conformal version of the so-called \emph{Bertrand-Darboux condition} \cite{bertrand_1857,darboux_1901}
\begin{equation}\label{eq:BD}
	d(\hat K(dV))+V\,d\rho+dV\wedge\rho=0\,.
\end{equation}
In other words, once we have specified an element $H\in\mathcal H$ and a trace-free conformal Killing tensor such that \eqref{eq:BD} holds, we may integrate for $W$. We introduce the space
\begin{align*}
	\mathcal K:=\{K\in\Gamma(\mathrm{Sym}^2_\circ(T^*M)\,|\, & \text{$K$ is a trace-free conformal Killing tensor}
	\\
	 &\quad\text{such that \eqref{eq:BD} holds for all $H\in\mathcal H$}\,\}.
\end{align*}
Given an element of $\mathcal K$, we are hence able to recover the integrals for the multi-Hamiltonian $\mathcal H$, i.e.\ we obtain the integrals for any choice of $H\in\mathcal H$.
This ensures that the concept of second-order (maximal conformal) superintegrability is well-defined for a multi-Hamiltonian. If the multi-Hamiltonian is non-degenerate, we say that the pair $(\mathcal H,\mathcal K)$ is a non-degenerate (second-order maximally conformally) superintegrable system.

Non-degenerate superintegrable systems have been studied intensively in the literature, cf.\ \cite{KKM2018} for an overview of many results on the subject. In \cite{KSV2023,KSV2024,KSV2024_bauhaus}, a geometric approach is pursued that is particularly suited for higher dimensions.
Specifically, assume that $\mathcal K$ is \emph{irreducible}, i.e.\ that the space of associated endomorphisms,
\[
	\mathcal K^\sharp := \{ A\in\mathrm{End}(TM)\,|\,g(A(X),Y)=K(X,Y) \text{ for some $K\in\mathcal K$}\,\}\,,
\]
is irreducible, in the sense that pointwise on an open and dense subset of $M$ they do not have a non-trivial invariant subspace in common.
Then it is shown in \cite{KSV2024} that we may write
\begin{equation}\label{eq:Wilcz}
	\nabla^2_{ij}V=\sum_{k=1}^n T\indices{_{ij}^k}\nabla_kV+\tau_{ij}\,V
\end{equation}
for every $H=\sum_{i,j=1}^n g^{ij}p_ip_j+V\in\mathcal H$. The point here is that the coefficients $T\indices{_{ij}^k}$ and $\tau_{ij}$ do not depend on the choice of $H$, but only on $\mathcal K$. Moreover, the coefficients form tensor fields $\hat T\in\Gamma(\mathrm{Sym}^2_\circ(T^*M)\otimes TM)$ and $\tau\in\Gamma(\mathrm{Sym}^2_\circ(T^*M))$, respectively. These tensor fields are called the \emph{primary} and, respectively, \emph{secondary structure tensor} of the non-degenerate superintegrable system. From the integrability conditions of \eqref{eq:Wilcz} it follows that
\begin{equation*}
	T_{ijk}=S_{ijk}+\bar t_ig_{jk}+\bar t_jg_{ik}+\bar t_kg_{ij}
\end{equation*}
where $S_{ijk}$ and $\bar t_i$ are components of a totally symmetric and trace-free tensor field $S\in\Gamma(\mathrm{Sym}^3_\circ)$ and of a closed $1$-form $\bar t\in\Omega^1(M)$, respectively.
Locally, $\bar t$ is therefore the differential of a function, say $\bar t=dt$, which we will draw upon in the following.

We now introduce the concept of abundant superintegrable systems.
As mentioned earlier, this concept can be traced back to \cite{KKM-3}, from which it follows that all $3$-dimensional non-degenerate superintegrable systems are abundant, i.e.\ they admit an additional (linearly independent) integral of the motion.

\begin{defn}\label{defn:abundant.system}
	A non-degenerate (second-order maximally conformally) superintegrable system is said to be \emph{abundant} if $\mathcal K$ has dimension $\frac{1}{2}\,n(n+1)-1=\frac12\,(n-1)(n+2)$.
\end{defn}

\subsubsection{Abundant structures}\label{sec:sbundant.structure}

Abundant superintegrable systems are a convenient model case for the study of non-degenerate systems because all known examples of non-degenerate superintegrable systems are, in fact, abundant.
Following \cite{KSV2023,KSV2024,KSV2024_bauhaus} and \cite{CV2025}, we will now introduce a geometric structure, called \emph{abundant structure}, to encapsulates an abundant superintegrable system in a geometric manner. It is based on the structural equations determined in \cite{KSV2024} and \cite{KSV2024_bauhaus} for abundant non-degenerate superintegrable systems in dimensions $n\geq3$ and, respectively, $n=2$. The relationship of Definition \ref{defn:abundant.system} above and Definition~\ref{defn:abudant.structures} below is as follows:
\begin{itemize}
	\item for a given abundant non-degenerate superintegrable system, the tensors $S$, $\bar t$ and $\tau$ satisfy (locally, where $\bar t=dt$ for a function $t$) the conditions in Definition \ref{defn:abudant.structures}.
	\item for a given abundant structure, we can solve (by integration) the equations\footnote{In the case $n\geq3$, we add Equation~\eqref{eq:tau.n>2} below to the system \eqref{eq:n>2.conditions} in order to define $\tau$.} in Definition \ref{defn:abudant.structures} for unique solutions $S$, $\bar t$ and $\tau$ in a neighborhood of a given point $x\in M$. We can then integrate \eqref{eq:Wilcz} for $V$ up to the choice of $n+2$ integration constants. It was proven in \cite{KSV2024} that for a generic choice of such a solution we can then integrate \eqref{eq:BD} for $K$ (subject to $\frac12(n-1)(n+2)$ integration constants), and that, from the resulting space~$\mathcal K$, we can choose $2n-2$ elements $K^{(\alpha)}$ such that $(H,F^{(1)},\dots,F^{(2n-2)})$ is a functionally independent collection of conformal integrals of the motion, where $F^{(\alpha)}=\sum_{i,j=1}^n K^{(\alpha)\,ij}p_ip_j+W^{(\alpha)}$ with $W^{(\alpha)}$ being the corresponding solution of~\eqref{eq:pre-BD}.
\end{itemize}
Equivalently, we may say that locally the concept of abundant structures and of abundant non-degenerate superintegrable systems coincide.

We proceed to state the definition of an \emph{abundant structure}, for which we introduce the following notation. We denote the Ricci tensor of the Riemannian metric $g$ by $\mathrm {Ric}$,
\[
	\mathrm{Ric}(X,Y)=\tr(R^g(\cdot,X)Y),
\]
where $R^g$ is the curvature tensor of $g$. The scalar curvature of $g$ is going to be denoted by $R$ as there is no real risk of confusion with the curvature tensor.
The Schouten tensor of $g$ is going to be denoted by $\P$, i.e.
\[
	\P = \frac{1}{n-2}\left( \mathrm{Ric}-\frac{R}{2(n-1)}g\right)
\]
Moreover, we denote the norm (with respect to the metric $g$) of a tensor with components $T_{{i_1}\cdots{i_r}}$ by
\[
|T|_g^2=T^{{i_1}\cdots{i_r}}T_{{i_1}\cdots{i_r}}\,.
\]
We may omit the subscript $g$, if the underlying metric is clear.
Also, for better legibility, we abbreviate $t_i:=\nabla_it$ in the following, leading to the natural equality $\bar t_i=\nabla_it=t_i$.

In the remainder of the paper, we furthermore adopt Einstein's summation convention for repeated pairs of (upper and lower) indices.

\begin{defn}[\cite{KSV2024,KSV2024_bauhaus,CV2025}]\label{defn:abudant.structures}
	Let $(M,g)$ be a conformally flat Riemannian manifold of dimension $n\geq2$, endowed with a totally symmetric and trace-free tensor field $S\in\Gamma(\mathrm{Sym}^3_\circ(T^*M))$ and a function $t\in\mathcal C^\infty(M)$ such that\footnote{Compared to~\cite{KSV2024_bauhaus}, two typographical errors have been corrected in~\eqref{eq:DXi.general} and~\eqref{eq:DDt.general}, involving the replacement of two erroneous signs.}
	\begin{itemize}
		\item if the dimension is $n\geq3$:
		\begin{subequations}\label{eq:n>2.conditions}
			\begin{align}
				\nabla^2_{ij}t
				&=  3\,\P_{ij}
				+\frac1{3(n-2)} \left(
				{S_{i}}^{ab}S_{jab}
				+\frac{n-6}{2(n-1)(n+2)}\,|S|^2\,g
				\right)
				\nonumber
				\\ &\qquad\qquad\qquad
				+\frac13 \left(
				t_i t_j-\frac12|\grad_gt|^2\,g
				\right)
				\\
				\nabla_lS_{ijk}
				&= \tfrac{1}{3}\,\Pi_{(ijk)\circ}
				\biggl(
				\,S\indices{_{il}^{a}} S_{j k a}
				+3\,S_{i j l} t_{k}
				+S_{i j k} t_{l}
				\nonumber
				\\ &\qquad\qquad\qquad
				+\left(
				\tfrac{4}{n-2}\,S\indices{_j^{ab}} S_{kab}
				-3\,S_{j k a} \bar{t}^{a}
				\right)\,g_{i l}
				\biggr)
				\\
				0 &= \Pi_\circ \left( g^{ab}S_{aik}S_{bjl}-g^{ab}S_{ail}S_{bjk} \right)
			\end{align}
		\end{subequations}
		\item if the dimension is $n=2$:
		\begin{subequations}\label{eq:n=2.conditions}
			\begin{align}
				\nabla_lS_{ijk}
				&=  \Pi_{(ijk)\circ}\left[
				-\frac23S_{ijk}t_l+2t_iS_{jkl}
				+\Xi_{ij}g_{kl}
				\right]
				\label{eq:DS.general}
				\\
				\nabla_k\Xi_{ij}
				&= \frac13S_{ijk}\,\,|S|_g^2
				+\frac43\Pi_{(ijk)\circ}\Xi_{jk}t_i
				+\frac43\Pi_{(ij)\circ} g_{jk}\beta_{i}
				\label{eq:DXi.general}
				\\
				\nabla^a\tau_{ak}
				&= -S_{kab}\tau^{ab}+\beta_{k}
				-\frac23\,\Xi_{ka}t^a - \frac49\,S_{kab}t^at^b
				- \frac59\,|S|_g^2\,t_k
				\nonumber
				\\ &\hspace*{6cm}
				+\frac12\nabla_kR- Rt_k
				\label{eq:divTau.general}
				\\
				\nabla^2_{ij}t
				&= \frac32\left[
				\tau_{ij}+\frac23S_{ija}t^a+\frac89\Pi_{(ij)\circ} t_it_j
				-\frac13\Xi_{ij}
				\right]
				\nonumber
				\\ &\qquad\qquad\qquad\qquad\qquad\qquad\qquad\qquad
				+\frac12g_{ij}\left[ \frac13\,|S|_g^2+\frac32R \right]
				\label{eq:DDt.general}
			\end{align}
		\end{subequations}
		where for better legibility we introduce tensor fields $\Xi\in\Gamma(\mathrm{Sym}^2_\circ(T^*M))$, $\tau\in\Gamma(\mathrm{Sym}^2_\circ(T^*M))$ and the $1$-form $\beta\in\Omega^1(M)$, defined by
		\begin{align*}
			\Xi_{ij}  &= \nabla^aS_{aij}+\frac23\,S_{ija}t^a\,, \\
			\tau_{ij} &= \frac23\,(\nabla^2_{ij}t-\tfrac1n\,\Delta t\,g_{ij})
			-\frac23\,S_{ija}t^a+\frac13\,\Xi_{ij}-\frac89(t_it_j-\tfrac1n\,t^at_a\,g_{ij})\,, \\
			\beta_k   &= S_{kab}\Xi^{ab}\,,
		\end{align*}
		where $\Delta$ is the Laplace-Beltrami operator.
	\end{itemize}
	We then say that $(S,t)$ is an \emph{abundant structure} on $(M,g)$. If we consider only the underlying smooth manifold $M$ to be fixed, we call $(g,S,t)$ an \emph{abundant structure} on $M$.
\end{defn}

Note that the equations in Definition~\ref{defn:abudant.structures} depend on $t$ only via $t_i=\bar t_i$.
Next, observe that the tensor fields $\Xi,\tau\in\Gamma(\mathrm{Sym}^2_\circ(T^*M))$ are mainly introduced as auxiliary objects, and indeed we shall use them as mostly technical definitions. However, we comment that $\tau$ has a deeper significance. In fact it distinguishes properly and conformally superintegrable systems, cf.~\cite{KSV2024,BLMS24} and also \cite{GKL2024}: if $\tau_{ij}=0$, then the abundant conformally superintegrable system admits proper integrals of the motion (rather than just on the zero locus of $H$). For completeness, we remark that for $n\geq3$, this tensor is given by
\begin{equation}\label{eq:tau.n>2}
	\tau_{ij} = \frac{2}{3(n-2)}\bigg(
					S_{ija}t^a+\Pi_{(ij)\circ} t_it_j -\Pi_{(ij)\circ} S\indices{_i^{ab}}S_{jab} 
			\bigg)
				-2\,\mathring{\mathsf{P}}_{ij}.
\end{equation}
The relationship between abundant structures and superintegrable systems has been elucidated by \cite{KSV2024,KSV2024_bauhaus}, and we refer the interested reader to these references for more details: Here we confine ourselves to mentioning that given an abundant manifold, one may integrate a certain system of partial differential equations whose solutions solve~\eqref{eq:integral}.
In fact, one typically obtains $\frac12(n-1)(n+2)$ linearly independent functions $F^{(\alpha)}$ of the form~\eqref{eq:second-order} that satisfy~\eqref{eq:integral} for a non-degenerate multi-Hamiltonian. The interested reader will find more details on this construction in~\cite{KSV2024,KSV2024_bauhaus} and \cite{KKM2018} as well as in the references therein.
In particular, it is shown in \cite{KSV2024} that the equations in Definition~\ref{defn:abudant.structures} are conformally invariant. By this we mean that these equations remain true if we replace $g\mapsto\Omega^2g$ for some nowhere vanishing $\Omega\in\mathcal C^\infty(M)$ (and its Levi-Civita connection accordingly) as well as
\[ S\to\Omega^2S\,,\qquad t\to t-3\ln|\Omega|\,. \]

\begin{defn}
	Two abundant structures $(g,S,t)$ and $(g',S',t')$ on the same smooth manifold are said to be \emph{conformally related} if $g'=\Omega^2g$, $S'=\Omega^2S$ and $t'=t-3\ln|\Omega|+c$ for some $c\in\mathbb R$.
\end{defn}

This and the next definition follow \cite{CV2025}, cf.\ \cite{KSV2024,KSV2024_bauhaus}.

\begin{defn}\label{defn:equivalence}
	Two abundant structures $(g,S,t)$ and $(g',S',t')$ on the same smooth manifold are said to be \emph{equivalent} if $dt'=dt$. 
\end{defn}

\section{Abundant structures and Weylian manifolds}\label{sec:main}

Having introduced abundant structures and Weylian manifolds, we are now going to relate these two structures.
Recall that for an exact Weyl structure $(M,c,D)$, the $1$-form $\theta_g$ in \eqref{eq:Weylian.connection} is exact for any $g\in c$.
We thus have $\theta_g=d\Theta_g$ for some function $\Theta_g\in\mathcal C^\infty(M)$. Analogously, consider the corresponding Weylian structure $(M,\Phi)$.
For $(g,\varphi)\in\Phi$, the $1$-form $\varphi$ is also exact, and $\varphi=d\phi$ for some function $\phi$, cf.\ Remark~\ref{rmk:Weylian.correspondence}. We say that $\Theta_g$ and $\phi$ are \emph{potentials} for $\theta_g$ and $\varphi$, respectively. With this terminology, we state the following lemma, which allows us to define abundant structures on Weylian manifolds. Its proof is implied by~\cite{KSV2024}.
\begin{lemma}\label{lem:construction.Weyl.structure}
	Let $(M,\Phi)$ be an exact Weylian structure. Let $(g,\varphi)\in\Phi$ and $\phi$ a potential for $\varphi$, $\varphi=d\phi$. Assume that $(M,g)$ admits an abundant structure encoded in the tensor field $S\in\Gamma(\mathrm{Sym}^3_\circ(T^*M))$ and the function $t\in\mathcal C^\infty(M)$.
	Next, let $(g',\varphi')$ be another representative of the Weylian metric $\Phi$, with potential $\phi'$, $\varphi'=d\phi'$.
	
	Then $(M,g')$ carries the abundant structure encoded in the structure tensor $S'=g'g^{-1}S$ and the function $t'=t-3(\phi-\phi')$.
\end{lemma}
Note that $g'g^{-1}S$ here means that we first raise, then lower one index using the metrics $g$ and $g'$, respectively.
We also remark that $dt'=dt-3(\varphi-\varphi')$, and hence the abundant structure is determined by $\varphi'-\varphi$ up to equivalence in the sense of Definition~\ref{defn:equivalence}. Conversely, if $g'=g$ and $t'-t$ is constant, it must hold that $\varphi'=\varphi$. The definite article for the abundant structure in the last sentence of the lemma is therefore justified.
\begin{proof}[Proof of Lemma~\ref{lem:construction.Weyl.structure}]
	Clearly, $(g',\varphi')=(\Omega^2g,\varphi-d\ln|\Omega|)$ for some function $\Omega$.
	We hence have $S'=\Omega^2S$ and $t'=t-3\ln|\Omega|$.
	It then follows from~\cite{KSV2024} that $(M,g',S',t')$ is an abundant structure, given that $(M,g,S,t)$ is one, and since all objects satisfy the properties of conformally related systems.
	The claim follows.
\end{proof}

We observe that $\hat S=g^{-1}S\in\Gamma(\mathrm{Sym}^2_\circ(T^*M)\otimes TM)$ and $\hat S=(g')^{-1}S'\in\Gamma(\mathrm{Sym}^2_\circ(T^*M)\otimes TM)$ coincide.
We also note that in the lemma, $t$ and $\varphi$ transform into $t'$ and $\varphi'$, respectively, in an analogous manner. The equation $\bar t-3\varphi=\bar t'-3\varphi'$ hence yields a conformal invariant, where $\bar t=dt,\bar t'=dt'$. It seems that this invariant was first observed in \cite{Capel_phdthesis} in the case $n=3$. We are now ready to define abundant structures on Weylian manifolds.

\begin{defn}\label{defn:weylian.abundant.structure}
	We call a triple $(M,\Phi,\hat S)$ consisting of a smooth manifold $M$ with Weylian metric $\Phi$ as well as a (2,1)-tensor field $\hat S\in\Gamma(\mathrm{Sym}^3_\circ(T^*M)\otimes TM)$ a \emph{natural abundant Weylian structure} if $S=g\hat S\in\Gamma(\mathrm{Sym}^3_\circ(T^*M))$ is trace-free and  satisfies the conditions in Definition~\ref{defn:abudant.structures} with $dt=3\varphi$.
\end{defn}
We introduce \emph{natural abundant Weyl structures} and \emph{natural abundant Weyl\-ian manifolds} accordingly, in the obvious way.
Note that in the definition, we suppress for the sake of conciseness the irrelevant constant that may be added to~$t$, cf.\ Definition~\ref{defn:equivalence}.

\begin{rmk}
	Compare Definition~\ref{defn:weylian.abundant.structure} to the definition of \emph{c-superintegrable systems}, cf.\ Definition~3.9 in~\cite{KSV2024}. We note that our focus here is on the case of abundant structures, but that the underpinning ideas are nontheless of a general nature.
	Indeed, we observe the fundamental difference that, in the reference, a c-superintegrable system is defined on a merely conformal structure $(M,c)$, i.e.\ a smooth manifold with a conformal class of metrics. The discussion in the current section describes the underpinning Weylian geometry. In~\cite{KSV2023}, this underlying Weylian structure exists implicitly, cf.\ Section~6 of \cite{KSV2024}. For $n\geq3$, it is indirectly inferred, by Eq.~(6.6) of~\cite{KSV2023}, from a conformal class of abundant superintegrable systems.
	
	We may reinterpret this as having subsumed the invariant difference function $\mathfrak t=t-3\phi$ bridging between the representative $(g,\varphi)\in\Phi$ of the Weylian metric and the abundant structure $(M,g,S,t)$.
	Indeed, if $\mathfrak t\ne0$, we may replace~$\Phi$ by a new Weylian metric $\Phi'=[(g,\varphi+\frac13\,d\mathfrak t)]$. In terms of the corresponding Weyl structure this presents itself as a conformal change of the Weylian connection. This observation now allows us to extend abundant structures to pre-fixed Weylian manifolds.
\end{rmk}

In the light of the remark, let the Weylian manifold be fixed, e.g.\ in terms of a Weyl structure $(M,g,D)$. We may then define an \emph{abundant structure} on this Weylian manifold as follows.
\begin{defn}\label{defn:abundant.structure.weylian.mfd}
	Let $(M,\Phi)$ be a Weylian structure. We call $(\hat S,\mathfrak t)$ consisting of a (2,1)-tensor field $\hat S\in\Gamma(\mathrm{Sym}^2_\circ(T^*M)\otimes TM)$ and a function $\mathfrak t\in\mathcal C^\infty(M)$ an \emph{abundant structure} on $(M,\Phi)$ if $(M,[(g,\varphi+\frac13\,d\mathfrak t)],\hat S)$ is an abundant Weylian manifold.
\end{defn}

We hence obtain:
\begin{thm}\label{thm:rebase}
	\begin{enumerate}[label=(\roman*)]
		\item An abundant structure $(\hat S,\mathfrak t)$ on a Weylian manifold $(M,\Phi)$ is a natural abundant Weylian structure precisely if $\mathfrak t$ is constant.
		\item If $(\hat S,\mathfrak t)$ is an abundant structure on the Weylian manifold $(M,\Phi)$, then $(M,[(g,\varphi+\frac13\,d\mathfrak t)],\hat S)$ is a natural abundant Weylian manifold, $(g,\varphi)\in\Phi$.
	\end{enumerate}
\end{thm}
\begin{proof}
	For part (i), observe that the structural equations~\eqref{eq:n>2.conditions} and~\eqref{eq:n=2.conditions}, respectively, involve only derivatives of the function $t$.
\end{proof}

Recall from Remark~\ref{rmk:Weylian.correspondence} that a Weyl structure can always be modified by a change of the connection $D\to D+\alpha\otimes\mathrm{id}$ for a $1$-form $\alpha\in\Omega^1(M)$. We see the analogous freedom in the second part of Theorem~\ref{thm:rebase}: we can realign the Weyl structure relative to the superintegrable system as long as we preserve exactness of the Weyl structure. The first part of the theorem, however, tells us that, up to the equivalence stated in Definition~\ref{defn:equivalence}, there is a unique natural base alignment of the Weyl structure relative to the superintegrable system, namely the one that makes $\mathfrak t$ a constant.

\subsection{Structural equations for \texorpdfstring{$n\geq3$}{n>=3}}\label{sec:structural.n>2}

Our aim is now to give a representative-free characterisation of abundant structures on Weylian manifolds. To this end, we need to establish some terminology for a given Weyl structure $(M,g,D)$.
We denote by $\Pi_{\mathrm{Sym}^3_0}$ the projection onto the trace-free totally symmetric component, and we introduce conformally adapted operators $\mathscr H:\mathcal C^\infty(M)\to\Gamma(\mathrm{Sym}^2_\circ(T^*M))$,
\begin{equation*}
	\mathscr H(f) = \mathrm{Hess}^{g}(f)-\frac1n\,g\,\Delta^g(f)-\mathring{\mathsf{P}}\,f\,,
\end{equation*}
and the conformal Laplacian 
\begin{equation*}
	\mathscr L({f}) = -4\,\frac{n-1}{n-2}\Delta^g{f}+\mathrm{Scal}^g{f}\,.
\end{equation*}

\begin{rmk}
	To ensure simplicity and accessibility, we apply the usual language of tensor fields and functions on $M$, but mention that an alternative formulation of the structural equations can be put forth in the language of weighted density bundles.
	The (weight-$w$) density line bundle $L^w$ is the bundle whose fibres (at a point $x\in M$) are given by the maps $\ell:\Omega^n(M)\to\R$ such that, for $r\in\R\setminus\{0\}$, $\ell(r\alpha)=|r|^{-\frac{w}{n}}\ell(\alpha)$. Similarly, one introduces the weighted tensor bundles $L^{w,k,k^*}=L^w\otimes (TM)^{k}\otimes(T^*M)^{k^*}$, for $w,k,k^*\in\mathbb{N}_0$. These are said to have weight $w+k-k^*$\cite{CP1999}.
	A (tensor) density $\mu$ of weight $\mathsf{w}$ thus transforms as $\mu\to|\Omega|^\mathsf{w}\mu$ under conformal rescalings.
\end{rmk}

We now introduce the conformally invariant sections $\mathcal P\in\Gamma(\mathrm{Sym}^2_\circ(T^*M))$, $\mathcal L\in\mathcal C^\infty(M)$, for a given Weylian structure $(M,\Phi=[(g,\varphi)])$,
\begin{align*}
	\mathcal P&:=e^{3\varphi}\mathscr H\left(e^{-3\varphi}\right)
	\\
	\mathcal L&:=e^{3(1-\frac{n}{2})\varphi}\mathscr L\left(e^{-3(1-\frac{n}{2})\varphi}\right)\,
\end{align*}
We furthermore introduce the conformally invariant tensor field
\begin{equation*}
	\mathscr S(X,Y) = \tr(\hat S(X,\hat S(Y,-)))
\end{equation*}
and its (conformally invariant) trace-free part
\begin{equation*}
	\mathring{\mathscr S} = \mathscr S-\frac1n\,g\,\tr_g(\mathscr S)\,.
\end{equation*}
Next, we introduce 
\begin{equation*}
	\mathfrak S(X,Y,Z)=\hat S(\hat S(Y,Z),X),
\end{equation*}
for $X,Y,Z\in\mathfrak X(M)$, and then a tensor field $\mathfrak A\in\Gamma(T^*M^{\otimes3}\otimes TM)$ via
\begin{equation*}
	{\mathfrak A}\indices{_{ijk}^l}
	=\frac12\,\Pi_{(ijk)\circ}\left(
	{\mathfrak S}\indices{_{ijk}^l}-\frac4{n-2}\mathscr S_{jk}g\indices{_i^l}
	\right)\,.
\end{equation*}
Note that $\mathfrak A$ is conformally invariant.
Let $\nabla^{\hat S}=\nabla^g-\hat S$ and denote its curvature and Ricci tensor by $R^S$ and $\mathrm{Ric}^S$, respectively. We then define the Weyl tensor $\mathrm{Weyl}^{\hat S}$ of $\nabla^{\hat S}$ (with respect to $g$) by
\[
\mathrm{Weyl}^{\hat S}(X,Y,Z,W)=g(R^{\hat S}(X,Y)Z,W)-(\mathsf{P}^{\hat S}\owedge g)(X,Y,Z,W)\,,
\]
where $\owedge$ is the Kulkarni-Nomizu product, and where $\mathsf{P}^S$ denotes the Schouten tensor of $\nabla^{\hat S}$ (with respect to $g$), i.e.
\[
\mathsf{P}^{\hat S}=\frac{1}{n-2}\left(\mathrm{Ric}^{\hat S}-\frac{\tr_g\mathrm{Ric}^{\hat S}}{2n(n-1)}\,g\right)\,.
\]
We are now ready to characterise abundant Weylian structures.

\begin{thm}\label{thm:abundant-superintegrable}
	Let $(M,\Phi)$ be a flat Weylian structure, $n\geq3$. Then $(M,\Phi,\hat S)$, where $\hat S$ is a $(1,2)$-tensor field, is \emph{abundant}, if and only if, for $\alpha\in\Omega^1(M)$ and $X\in\mathfrak X(M)$,
	\begin{align}
		\mathcal P &= -\frac1{9(n-2)}\,\mathring{\mathscr S}
		\label{eq:invariant.n>2.P}
		\\
		\mathcal L &= -\frac29\frac{3n+2}{n+2}\,|S|^2
		\label{eq:invariant.n>2.L}
		\\
		D_X\hat S &= \mathfrak A(X,\cdot,\cdot)
		\label{eq:invariant.n>2.S}
		\\
		\mathrm{Weyl}^{\hat S} &= 0
		\label{eq:invariant.n>2.Weyl}
	\end{align}
\end{thm}
Note that \eqref{eq:invariant.n>2.P}--\eqref{eq:invariant.n>2.Weyl} are conformally invariant conditions, i.e.\ that they hold true also after a conformal rescaling.
\begin{proof}
	The conditions \eqref{eq:invariant.n>2.P}--\eqref{eq:invariant.n>2.Weyl} are obviously equivalent to \eqref{eq:n>2.conditions}.
	Hence, for each $(g,\varphi=d\phi)\in\Phi$, $(M,g,g\hat S,t)$, $t=\frac13\phi$, defines an abundant manifold. Hence, $(M,\Phi,\hat S)$ defines an abundant Weylian manifold.
\end{proof}

\noindent An analogous statement holds for abundant structures (with non-vanishing $\mathfrak t$) on flat Weylian structures, the details of which we leave to the interested reader.

\subsection{Structural equations for \texorpdfstring{$n=2$}{n=2}}\label{sec:structural.n=2}

Structural equations can also be written down, in conformally invariant form, for $2$-dimensional abundant Weylian structures.
We introduce the tensor section $\mathfrak Z\in\Gamma(\mathrm{Sym}^2_\circ(T^*M)\otimes T^*M\otimes TM)$ via
\[
\mathfrak Z_{ijl}^k = g^{ka}\Pi_{(ija)\circ}\left( \Xi_{ij}g_{al}\right)
\]
and observe that it is conformally invariant. Moreover, we introduce $\mathfrak b\in\Gamma(\mathrm{Sym}^2_\circ(T^*M)\otimes T^*M)$ by
\[
\mathfrak b_{ijk} = \Pi_{(ij)\circ} \beta_{i}g_{jk}\,,
\]
which we likewise observe to be conformally invariant.
Reconsidering the relevant equations in Definition~\eqref{defn:abudant.structures}, we now find the following theorem.
\begin{thm}
	Let $(M,\Phi)$ be a flat Weylian structure, $n=2$. Then $(M,\Phi,\hat S)$, where $\hat S$ is a $(1,2)$-tensor field, is \emph{abundant}, if and only if, for $X\in\mathfrak X(M)$,
	\begin{align*}
		D_X\hat S
		&=  \mathfrak Z(\cdot,\cdot,X)
		\\
		D_X\Xi
		&= \frac{|S|_g^2}{3}S(X,\cdot,\cdot)
		+\frac43 \mathfrak b(\cdot,\cdot,X)
		\\
		\mathcal L^\mathrm{2D}
		&= \frac19\,|S|^2
	\end{align*}
	where we introduce
	\[
	\mathcal L^\mathrm{2D}:=\left( \Delta-\frac32R \right)t\,.
	\]
\end{thm}

These conditions are conformally invariant.
Note that $\mathcal L^\mathrm{2D}$ behaves like a density of weight $q=-2$, i.e.\ $\mathcal L^\mathrm{2D}\mapsto \Omega^{-2}\mathcal L^\mathrm{2D}$, under conformal changes $g\mapsto\Omega^2g$.
Indeed, letting $t'=t-\ln|\Omega|$,
\begin{align*}
	\Delta t
	&\mapsto \Omega^{-2}\left[
	\Delta t'-2\,dt'(\grad_g\ln|\Omega|)+2\,g^{-1}(d\ln|\Omega|,dt')
	\right]\,,
	\\
	R &\mapsto \Omega^{-2}\left[ R-2\Delta\ln|\Omega| \right]\,,
	\\
	\intertext{and hence}
	\Delta t-\frac32\,R
	&\mapsto \Omega^{-2}\left[ \Delta (t-3\ln|\Omega|)-\frac32\,(R-2\,\Delta\ln|\Omega|) \right]
	\\
	&= \Omega^{-2}\left[ \Delta t-3\Delta\ln|\Omega|-\frac32\,R +3\,\Delta\ln|\Omega| \right]
	= \Omega^{-2}\left[ \Delta t-\frac32\,R \right]\,.
\end{align*}

\subsection{Semi-Weyl manifolds}\label{sec:semi-weylian}

Re-examining the geometric structure of abundant superintegrable systems, we have explicitly described the underlying Weyl structure. This allowed us to encode, on a Weyl structure $(M,c,D)$, conformal classes of abundant structures in an invariant tensor field $\hat S\in\Gamma(\mathrm{Sym}^2_\circ(T^*M)\otimes TM)$, corresponding for $g\in c$ to a totally symmetric and trace-free tensor field $S:=g\hat S\in\Gamma(\mathrm{Sym}^3_\circ(T^*M))$.
Consider an abundant structure $(M,g,S,t)$. Then the corresponding natural Weyl structure is given by the representative $(g_0,0)$ where $g_0=e^{-\frac23 t}$. Indeed, recalling Lemma~\ref{lem:construction.Weyl.structure}, a representative of the Weylian metric is given by
\[
	\left( \Omega^2g,-\frac13\,dt-d\ln|\Omega| \right)
\]
where $\Omega$ is a smooth function. Asking that $dt-d\ln|\Omega|=0$ (we recall that the Weylian manifold is exact), we obtain $\Omega^2=\exp(-\frac23t)$, modulo multiplication by an irrelevant constant. It follows that, letting $g_0:=e^{-\frac23 t}g$,
\[
	(g_0,0)\sim\left(g,-\frac13 dt\right).
\]
We hence have
\[
	Dg_0=0
\]
where $D$ is the Weyl connection associated to the Weylian metric $\Phi:=[(g,-\frac13 dt)]$, compare Definition~\ref{defn:Weyl.structure}. This shows:
\begin{thm}\label{thm:Weyl.structure.LC.g0}
	The natural abundant Weylian structure (cf.\ Definition~\ref{defn:weylian.abundant.structure}) of an abundant structure $(M,g,S,t)$ is precisely the Levi-Civita connection of the metric $g_0=\exp(-\frac23 t)g$.
\end{thm}
We remark that, in the terminology of \cite{KSV2024}, the metric $g_0$ is the metric of the ``system in standard scale'' associated to the system corresponding to the abundant structure $(M,g,S,t)$ of the theorem.
Changing our vantage point, we will now explain how this allows us to conceptualise the abundant structure as a semi-Weyl manifold. We find the following definition in~\cite{Matsuzoe2010}.
\begin{defn}
	A semi-Weyl structure is a quadruple $(M,g,\bar D,\theta)$ consisting of a Riemannian manifold $(M,g)$, a torsion-free affine connection $D$ on $M$, and a $1$-form $\theta\in\Omega^1(M)$, such that
	\[
		\bar Dg+\theta\otimes g\in\Gamma(\mathrm{Sym}^3(T^*M))
	\]
	is totally symmetric.
\end{defn}
Consider an abundant structure $(M,g,S,t)$ as before and set $g_0=\exp(-\frac23\,t)$, as suggested by Theorem~\ref{thm:Weyl.structure.LC.g0}. We denote the Levi-Civita connection of the metric~$g_0$ by~$D$.
We then define
\[
	\bar D_XY:=D_XY+\hat S(X,Y),
\]
for any $X,Y\in\mathfrak X(M)$, where $\hat S=g^{-1}S$.
A short computation yields
\[
	\bar Dg_0=-2(g_0\hat S),
\]
since
$ Z(g_0(X,Y)) = g_0(D_ZX,Y)+g_0(X,D_ZY) $
for $X,Y,Z\in\mathfrak X(M)$.
Thereby, the abundant structure $(M,g,S,t)$ gives rise to the semi-Weyl structure with data $(M,g_0,\bar D,0)$.
More generally, we obtain the following theorem.
\begin{thm}
	An abundant structure $(M,g,S,t)$ has an underlying semi-Weyl structure $(M,g,\bar D,-\frac23\,dt)$. In particular, denoting by $\nabla^0$ the Levi-Civita connection of $\exp(-\frac23 t)g=:g_0$, let
	\[
		\bar D:=\nabla^0+\hat S
	\]
	where $\hat S=g^{-1}S$.
	For $h=\Omega^2 g$ another metric in the same conformal class as~$g$, it then follows that $(M,h,\bar D,\theta:=-\frac23 dt-2d\ln|\Omega|)$ is a semi-Weyl structure, and
	\[
		\bar D_Zh(X,Y)+\theta(Z)h(X,Y) \in\Gamma(\mathrm{Sym}^3(T^*M)).
	\]
\end{thm}
\begin{proof}
	Without loss of generality, let $g=g_0$, such that $h=\Omega^2g_0$ and $t=0$. A straightforward computation shows
	\begin{multline*}
		\bar D_Z(\Omega^2g_0)-2\left(d\ln|\Omega|\right)(Z)\,\Omega^2\,g_0(X,Y)
		\\
		= 2\Omega\,d\Omega(Z)g_0(X,Y)+\Omega^2\bar D_Zg_0(X,Y)  -2\left(d\ln|\Omega|\right)(Z)\,\Omega^2\,g_0(X,Y)
		\\
		= -2\Omega^2 (g_0\hat S)(X,Y,Z).
	\end{multline*}
	The general case is proven analogously.
\end{proof}

We have hence found that the data of $(M,g,S,t)$ can alternatively be encoded in the semi-Weyl structure $(M,g,\bar D,-\frac23 dt)$ where we identify abundant structures $(g,S,t)\sim(g,S,t')$ when $t-t'$ is a constant, cf.\ Definition~\ref{defn:equivalence}.
We also mention that this endows $(M,g_0,g_0g^{-1}S)$ with the structure of a statistical manifold, cf.\ \cite{Matsuzoe2010}.
	
We conclude the paper with two examples.

\begin{ex}
	Consider the Riemannian metric $g=\sum_{i=1}^n dx^i\otimes dx^i$ on $M=\mathbb R^n$.
	We denote its Levi-Civita connection by $\nabla^g$.
	The non-degenerate harmonic oscillator system \cite{KKM2018} is given by the family of potentials
	\[
		V_\mu = \mu_0\sum_{i=1}^n (x^i)^2+\sum_{i=1}^n \mu_ix^i + \mu_{n+1},
	\]
	where $\mu=(\mu_0,\mu_1,\dots,\mu_{n+1})\in\mathbb R^{n+2}$.
	It is an abundant non-degenerate irreducible conformally superintegrable system with the structure tensors given by, cf.~\cite{KSV2023},
	\[
		T\indices{_{ij}^k}=0\,,\qquad \tau_{ij}=0\,.
	\]
	It follows that $\bar t=0$ and $\hat S=0$, and hence the natural abundant Weylian structure is 
	\[
		(M,\Phi,\hat S)=(\mathbb R^n,[(g,0)],0)\,,
	\]
	which corresponds to the Weyl structure $(\mathbb R^n,[g],\nabla^g)$.
	The associated semi-Weyl structure is thus given by
	\[
		(\mathbb R^n,g,\nabla^g,0)\,.
	\]
\end{ex}

\begin{ex}
	Consider the Riemannian metric $g=\sum_{i=1}^n dx^i\otimes dx^i$ on $M=\mathbb R_+^n$, with Levi-Civita connection $\nabla^g$.
	The Smorodinski-Winternitz system is given by the family of potentials
	\[
		V_\mu = \mu_0\sum_{i=1}^n (x^i)^2+\sum_{i=1}^n \frac{\mu_i}{(x^i)^2} + \mu_{n+1},
	\]
	where $\mu=(\mu_0,\mu_1,\dots,\mu_{n+1})\in\mathbb R^{n+2}$.
	It is an abundant non-degenerate irreducible conformally superintegrable system with the structure tensors given by
	\[
		T\indices{_{ij}^k}=\Pi_{(ij)\circ}C\indices{_{ij}^k}\,,\qquad \tau_{ij}=0,
	\]
	where
	\[
		C=-\sum_{i=1}^n\frac{3}{x^i} dx^i\otimes dx^i\otimes \frac{\partial}{\partial x^i}\,.
	\]
	It follows that $t=-\frac{3}{n+2}\sum_{i=1}^n \ln(x^i)$, up to an irrelevant constant \cite{KSV2023}. Denoting the Kronecker delta by $\delta_{ij}$, we then define $S\in\Gamma(\mathrm{Sym}^3_\circ(T^*M))$ by
	\[
		S_{ijk}:=-\frac{3}{x^i}\delta_{ij}\delta_{jk}
			+\frac{3}{n+2}\frac{1}{x^i}\delta_{jk}
			+\frac{3}{n+2}\frac{1}{x^j}\delta_{ik}
			+\frac{3}{n+2}\frac{1}{x^k}\delta_{ij}
	\]
	and subsequently introduce $\hat S\in\Gamma(\mathrm{Sym}^2_\circ(T^*M)\otimes TM)$ by virtue of the inverse metric $g^{-1}$.
	The associated natural abundant Weylian structure of the Smorodinski-Winternitz system thus is 
	\[
		\left( \mathbb R_+^n,\left[\left(g,-\frac{3}{n+2}\sum_i\frac{dx^i}{x^i}\right)\right], \hat S\right)\,.
	\]
	It corresponds to the Weyl structure $(\mathbb R_+^n,[g],D)$ where $D$ is the Levi-Civita connection of
	\[
		\exp\left(-\frac{2}{n+2}\sum_j\ln(x^j)\right)\,g
		=\left(\prod_{j=1}^{n}x^j\right)^{-\frac{2}{n+2}}\,\sum_{i=1}^ndx^i\otimes dx^i.
	\]
	The underlying semi-Weyl structure is finally obtained as
	\[
		\left( \mathbb R_+^n,g,D+\hat S,\frac{2}{n+2}\sum_{i=1}^n\frac{dx^i}{x^i} \right)\,.
	\]
\end{ex}
	
\section*{Acknowledgements}
The author thanks John Armstrong and Vicente Cortés for discussions.
This research was funded by the German Research Foundation (\emph{Deutsche Forschungsgemeinschaft}) through the Research Grant 540196982. The author also thanks the \emph{Forschungsfonds} of the Department of Mathematics at the University of Hamburg for support. 

\printbibliography
	
\end{document}